\documentclass[12pt,a4paper,oneside,openright]{elsarticle}

\usepackage[left=3.5cm,right=3cm,top=3cm,bottom=3cm]{geometry}

\author{Francesca Angrisani}

\usepackage[utf8]{inputenc}

\usepackage{amsmath}
\usepackage{amsfonts}
\usepackage{amssymb}
\usepackage{amsthm}
\usepackage{xcolor}
\usepackage{graphicx}
\usepackage[english]{babel}
\usepackage[T1]{fontenc}

\usepackage{graphicx}
\graphicspath{{./Images/}}

\newtheorem{thm}{Theorem}[]
\newtheorem{cor}[thm]{Corollary}
\newtheorem{lem}[thm]{Lemma}

\theoremstyle{definition}
\newtheorem{mydef}{Definition}[section]

\theoremstyle{remark}

\theoremstyle{remark}
\newtheorem{rmk}{Remark:}
\usepackage{setspace}
\def\Xint#1{\mathchoice 
	{\XXint\displaystyle\textstyle{#1}}%
	{\XXint\textstyle\scriptstyle{#1}}%
	{\XXint\scriptstyle\scriptscriptstyle{#1}}%
	{\XXint\scriptscriptstyle\scriptscriptstyle{#1}}%
	\!\int} 
\def\XXint#1#2#3{{\setbox0=\hbox{$#1{#2#3}{\int}$} 
		\vcenter{\hbox{$#2#3$}}\kern-.5\wd0}} 
\def\fint{\Xint -}
\begin{document} 

	\begin{center}
		\Large
{\bfseries A new norm on $BLO$, matters of approximability and duality}

\end{center}

\begin{center}
	FRANCESCA ANGRISANI \footnote{ \noindent Department of Information Engineering, Computer Science and Mathematics (DISIM), \\ Via Vetoio, Coppito 1, L'Aquila.  \\ francesca.angrisani@univaq.it \\ Present address: Viale di Augusto n.122 -80125 Napoli - since 15th July to 31st August: only via mail}
	
\end{center}

\normalsize

\noindent

\medskip \noindent
\textbf{Abstract -- } In this paper, we obtain an alternative expression for the distance of a function in $BLO$ from the subspace $L^\infty$. The distance is the one induced by choosing a new "norm" on $BLO$, equivalent to the usual one and that has the advantage to make the distance to $L^{\infty}$ explicitely and exactly computable. We also prove that this new norm has got the norm attaining property on $BLO$. \\ We address the issue of approximability by truncation as it is not obvious even in the closure of $L^\infty$ in $BLO$ that functions can be approximated by truncation. As a matter of fact, a few examples of this are presented. The easier issue of approximability by convolution is also addressed.\\
In the last section we finally prove that the "dual" of $VLO$ is isometric to $VMO^{\ast}$.\\
 \\
 \emph{Key words:} $BLO$, $L^\infty$, distance, truncation, approximability, duality.\\
 \\
 \emph{Mathematics Subject Classification:}  $46E30$  \\
	\section{Introduction}
	As it is well known, Lebesgue spaces $L^p(\Omega,\mu)$ and many of their generalizations (Grand Lebesgue, Lorentz, Orlicz, Marcinkiewicz spaces) are defined by conditions that, in some intuitive sense, limit the size of the function, but in a way such that the distribution of the values of the function is irrelevant.\\ In other words, one measures the mass of the function irrespective of where it is allocated. This is formalized by saying that these spaces are \emph{rearrangement invariant}, that is, given two functions $f$ and $g$ such that $$|\{x\in \Omega: |f(x)|\ge \lambda\}|=|\{x\in \Omega: |g(x)|\ge \lambda\}|, \quad \forall \lambda\ge0,$$ then $f$ belongs to the space if and only if $g$ does; two functions satisfying the above property are called equimeasurable. In the aforementioned Lebesgue spaces, for example, this is made clear by the so-called Cavalieri Principle.\\ If a function $f$ belongs to a given rearrangement invariant function space $X$, then the same can be said of its \emph{non-increasing rearrangement} $$f^*: t \in [0,|\Omega|]\mapsto \inf\Big\{\lambda: |\{x\in \Omega: |f(x)|\ge \lambda\}|\le t\Big\} \in \mathbb{R}^+,$$ which is a non-increasing function that is equimeasurable to the starting function $f$ (see \cite{Koren}).\\
	Not all function spaces, though, are rearrangement invariant.\\
	An interesting aspect of functions is the way they oscillate, which is neglected in rearrangement invariant spaces, is crucial to the definition of other interesting function spaces, as spaces defined by mean of oscillation. In the mathematical literature, a plethora of different meanings and formal definitions have been associated to the word "oscillation".\\ In reality in many cases spaces defined by means of oscillation, in many different sense of the word, fit into two modes: spaces in which oscillation is bounded and spaces in which oscillation is vanishing, i.e. arbitrarily small when measured on a sufficiently small set. This structure, called \emph{o-O type structure}, fits into a very general and abstract framework for couple of non-reflexive Banach function spaces, introduced by K.M. Perfect in 2013 (see \cite{Perfect1},\, \cite{Perfect2}, \,\cite{Perfect3}). He shows that an example of $(o,O)$ pair is $(BMO,VMO)$. \\
	Other examples of $(o,O)$ pair of functional spaces are: $(B,B_0)$ (\cite{Sbordone}) (B is  a very large space of functions, that generalize $BMO$ introduced by Brezis, Bourgain and Mironescu) \cite{B} (see also \cite{mosca} and \cite{Fusco} ) and $B_0$ an important subspace of $B$),  the well known $(Lip_{\alpha},lip_{\alpha})$ (see \cite{trottola}), Orlicz space with $\Delta_0$ condition and their {\em Morse subspace }$(M^{\psi},L^{\psi})$ (see \cite{orlicz}).\\
	 The framework introduced by Perfect often is also useful to obtain some information about dual spaces, distance formulas and result of atomic decomposition of the predual of the $O$-space (see for example \cite{fen}, \cite{schiattasbord}, \cite{gg}, \cite{Sbordone}, \cite{trottola}).\\
	   However space defined by means of oscillation has got also many application.
	 In fact, for example, $BMO$ was introduced by John and Nirenberg (\cite{lorigineditutto}) to approach a problem of elasticity theory, in particular to describe the  elastic strain. \\
	 In this paper we focus on two subcones of $BMO$ and $VMO$ called respectively $BLO$ and $VLO$. \\
	 Also this two cones are defined by mean of oscillation and it is natural to imagine the couple $(BLO,VLO)$ as an $o-O$ pair. \\ Unfortunately the theory introduced by Perfect can not be applied to these classes of functions, as they are only cones and not vectorial spaces. \\
	  $BMO$ and $BLO$ which are function classes defined by imposing a bound on two different kinds of oscillation, the so called mean oscillation of $f$ over $I$, i.e. 
	 $$\frac{1}{|I|}\displaystyle\int_I \left|f(x)-\frac{1}{|I|}\displaystyle\int_I f(y)\,dy\right|\,dx.$$
	 and the so called lower oscillation of $f$ over $I$, i.e.
	 $$\frac{1}{|I|}\displaystyle\int_I f(x)\,dx- \inf\limits_{x \in I}f(x)$$
	 respectively, where $I$ is an interval if we are dealing with functions of a single variable, while it is a hypercube with sides parallel to the coordinate hyperplanes if we wish to define the function class in $\mathbb{R}^n$.\\
   	In particular, in this paper we introduce a new equivalent norm in $BLO$, inspired by some characterizations of $BLO$ by Coifman and Rochberg in 1980 (see \cite{coif}), that is much more natural when addressing the aforementioned distance problem.\\
   	We also prove the our new norm has got the \emph{norm attaining property}, i.e. our norm is attained on the closed unit balls of  $BLO$. This property does not hold for classical norm on $BLO$ (it holds only on  $VLO$ functions \cite{giak}).\\
   		In the second part of paper we discuss the behaviour of $BLO$ functions with respect to truncation.\\ As the norm on $BLO$ is defined by means of oscillation, we may expect it to have strange behaviour with respect to truncation, which is a procedure acting on the size of the function.\\ That is exactly what happens: we show that bounded functions, functions that can be $BLO$-approximated by their truncations and functions that can be $BLO$-approximated by bounded functions are three completely different concepts, forming subspaces  $L^\infty,T,\overline{L^\infty}^{BLO}$ in strict consecutive inclusion: $$L^\infty \subsetneq T\subsetneq \overline{L^\infty}^{BLO}.$$ Examples are provided in the section \ref{approx}. \\
	Then we also deal the easier issue of approximation by convolution. \\
	Another reason for which $BMO$ is a very important space is that it plays the same role in the theory of Hardy spaces $H^p$ that the space $L^{\infty}$ of essentially bounded functions plays in the theory of $L^p$ spaces. \\
	In fact C. Fefferman and E.M. Stein in 1972 proved that $H^1$ is the predual of $BMO$ (see \cite{feff}). \\
	Instead it is still unknown what is the dual space of $BMO$. \\
	So, finally we characterize the "dual" of $VLO$ (cleary we will use a weak definition of dual of $VLO$): in particular we prove that it is isometric to $VMO^*$.
\section{Notation}
	\begin{mydef}
		We say that a real valued locally integrable function $f(x) \in L^1_{loc}(\mathbb{R})$ is of {\em Bounded Mean Oscillation} ($f(x) \in BMO(\mathbb{R})$) if:
		\begin{equation}
		\sup\limits_I \fint_I |f(x)-f_I|\,dx =\|f\|_{BMO}<\infty
		\end{equation}
		where $f_I$ denotes $\fint_I f(x)\,dx=\frac{1}{|I|}\int_I f(x)\,dx$ and $I\subset \subset \mathbb{R}$ is the generic compact interval.\\
	\end{mydef}
\noindent
	$BMO$, first introduced by John and Niremberg in 1961, is a vector space and $\|\cdot\|_{BMO}$ is a seminorm whose kernel is the set of all functions which are almost everywhere equal to a constant.\\ If we consider $BMO$ modulo this subspace, we get a Banach space.\\
	\begin{mydef} We say that a $BMO (\mathbb{R})$ function $f(x)$ is of {\em Vanishing Mean Oscillation} ($f \in VMO (\mathbb{R})$) if it also satisfies:
		\begin{equation}
		V(f)=\limsup\limits_{|I|\to 0}  \fint_I |f(x)-f_I|\,dx =0.
		\end{equation}
	\end{mydef}
\begin{mydef} We say that a real valued locally integrable function $f(x) \in L^1_{loc}(\mathbb{R})$ has {\em Bounded Lower Oscillation} ($f(x) \in BLO(\mathbb{R})$) if
	\begin{equation} \label{treno}
	\sup\limits_I \fint_I [f(x)-\inf_If]\,dx = \sup\limits_I [f_I-\inf_If]= \|f\|_{BLO}<\infty.
	\end{equation}
\end{mydef}
\noindent
As for $BMO$, we will think of the class of $BLO(\mathbb{R})$ functions modulo constants. 
\begin{mydef} We say that a $BLO(\mathbb{R})$ function has { \em Vanishing Lower Oscillation} ($f \in VLO(\mathbb{R})$) if it also satisfies:
	\begin{equation}
	W(f)=\limsup\limits_{|I|\to 0}  [f_I-\inf_If] =0.
	\end{equation}
\end{mydef}
Of course $\|\cdot\|_{BLO}$ vanishes only on the set of almost everywhere constant functions, is sub-additive and homogenous for positive scalars but $BLO\not =-BLO$\footnote{An example of a function $f(x)=-\log(|x|) \in BLO$ but not in $-BLO$ was pointed out by Korey \cite{KOR}}. For this reason and this reason only, $BLO$ is not a vector space (but only a cone) and $\|\cdot\|_{BLO}$ is not a norm on it in the classical sense. Even if it lacks homogeneity for negative constants, following other authors (see \cite{KOR} and \cite{coif}), we will call this quantity a norm on $BLO(\mathbb{R})$ anyways.\\
Also the class of all $VLO$ functions is not a vector space for the same reason.\\
Sometimes we will use the notation $BMO$, $VMO$ $BLO$ and $VLO$: we always intend $BMO(\mathbb{R})$, $VMO(\mathbb{R})$ $BLO(\mathbb{R})$ and $VLO(\mathbb{R})$.\\
It is obvious, by the definitions, that:
\begin{equation} \label{raddoppionorme}
\|f\|_{BMO}\le 2\|f\|_{BLO},\quad \|f\|_{BMO},\|f\|_{BLO}\le 2 \|f\|_{L^\infty} \quad \forall f \in L^1_{loc}. 
\end{equation}
Let us know introduce the $A_1$ class of Muckenhoupt. \\
\begin{mydef}
	A weight $w: \mathbb{R} \rightarrow [0, +\infty[$ belongs to the $A_1$ class of Muckenhoupt if:\\
	$$A_1(w)=\sup_{I \subset \subset \mathbb{R}} \frac{\fint_I w(x) dx}{inf_I w} < \infty$$ or equivalently: 
\end{mydef}
\noindent 
Inspired by a work of Coifman and Rochberg on a decomposition of $A_1$ weights through the Hardy-Littlewood maximal operator and by its consequences on $BLO$, we define a new, equivalent norm on BLO.\\
This new norm $\|\cdot\|_{BLO}'$ will have the advantage of making
\begin{equation*}
dist_{BLO'}(f,L^\infty):=\inf\limits_{g \in L^\infty}\|f-g\|_{BLO}'\footnote{The $||\cdot||_{BLO}'$ will be define in  \eqref{dog}.} 
\end{equation*} 
explicitely and exactly computable.\\
As a matter of fact, we will prove the equality:
\begin{equation}\label{tesi}
dist_{BLO'}(f,L^\infty)=\sigma(f):=\inf\{\mu>0:\ e^{\frac{f}{\mu}} \in A_1\}
\end{equation} 
showing that, in this new norm, the distance from $L^\infty$ is exactly the reciprocal of the critical exponent for which the function $e^{f}$ is in $A_1$. The right hand side in \eqref{tesi} is well defined for a known duality between $A_1$ and $BLO$ which is presented in the following section.
In a previous paper, (see \cite{mio}), we proved that
 \begin{thm}
	\label{theorema4}
	There exist two absolute constants $d_1,d_2>0$ such that for every real valued function $f \in BLO(\mathbb{R})$ the following inequalities hold:
	\begin{equation}
	d_1\sigma(f)\le dist_{BLO}(f,L^\infty) \le d_2\sigma(f)
	\end{equation}
	where:
	\begin{equation}
	\sigma(f)=\inf\{\mu>0:\ e^{\frac{f}{\mu}} \in A_1 \}
	\end{equation}
\end{thm} \noindent
In other contests similar results have been obtained also in \cite{Garnett}, \cite{SAR},  \cite{cap}, \cite{CAR}, \cite{nando} and \cite{form}. \\
Our new result will be a quantitative improvement of Theorem $\ref{theorema4}$ of \cite{mio}.
\section{Preliminaries}
We begin by quoting an important result concerning $BMO$.
\begin{lem}[{\bfseries John-Nirenberg inequality, \cite{lorigineditutto}}] \label{JNLEMMA}
	There exist two absolute constants $c_1,c_2>0$ such that, for every $f \in BMO$, $I \subset \mathbb{R}$ and $\lambda>0$ we have:
	\begin{equation} \label{JN}
	|\{t \in I:\ |f(t)-f_I|>\lambda\}|\le c_1e^{-c_2\lambda/\|f\|_{BMO}}|I|.
	\end{equation}
\end{lem}\noindent
\medskip
\medskip
There is a strong and useful connection between $BMO$ and $A_2$  (see \cite{coif} and also \cite{Sbord}). \\
This next lemma, due to Coifman and Rochberg, shows, instead, a less known connection between the Muckenhoupt class $A_1$ and $BLO$.
 
\begin{lem}[{\bfseries R.R. Coifman, R. Rochberg}] \label{dualita}
	Let $w \in L^1_{loc}(\mathbb{R})$. Then:
	\begin{itemize}
		\item $w \in A_1 \Rightarrow \log(w) \in BLO$ and $\|\log(w)\|_{BLO}\le \log(A_1(w))$
		\item $w \in BLO \Rightarrow \ e^{w/\mu}\in A_1, \quad \forall \mu>\frac{\|w\|_{BMO}}{c_2}$
	\end{itemize}
	where $c_2$ is the constant from (\ref{JN}) in Lemma \ref{JNLEMMA}.
	\end{lem}
\begin{proof}
	For the proof we refer to \cite{libroneverde},p. 157.
	\end{proof}

Let now introduce another important tool that we will use: \\
\begin{mydef}
\noindent The operator $M: L_{loc}^1 \rightarrow \mathbb{R}$ such that:
\begin{equation*}
Mw(x)=\sup\limits_{I \ni x} \fint_I |w(t)|dt
\end{equation*}
is called Hardy Littlewood maximal operator.
\end{mydef}
\noindent We quote another result Coifman and Rochberg, concerning powers of $Mw$ for $w \in L^1_{loc}(\mathbb{R})$.

\begin{lem}[{\bfseries R.R. Coifman, R. Rochberg,  \cite{coif}}]\label{powersofm}
	Let $w \in L^1_{loc}$ be a weight and $\varepsilon \in [0,1)$. Then:
	\begin{equation*}
	Mw(x)^\varepsilon \in A_1
	\end{equation*}
		with $A_1$ constant depending on $\varepsilon$ but not on $w$.
\end{lem}
\noindent The two authors also proved that, modulo $L^\infty$ functions, all $A_1$ functions arise as some power of $Mw$ for a suitable $w \in L^1_{loc}(\mathbb{R})$, i.e:
\begin{thm}[{\bfseries R.R. Coifman, R. Rochberg,  \cite{coif}}]\label{coifdec}
	Assume $w \in A_1$.\\ There are functions $0<A<b(x)<1 \in L^\infty$, $g \in L^1_{loc}$ and a number $\varepsilon \in [0,1)$ such that:
	\begin{equation*}
	w(x)=b(x)Mg(x)^\varepsilon
	\end{equation*}
\end{thm}
\noindent For the convenience of the reader we will repeat the proof of this theorem paying particular attention to the value of $A$ and how it depends on $A_1(w)$.\\
To do so, the so called reverse Holder inequality is needed:
\begin{thm} \label{a1}
	If $w \in A_1$, then there is a sufficiently small value of $\eta>0$ and a  constant $c_\eta$ depending on $\eta$ but not on $I \subset \mathbb{R}$ such that:
	\begin{equation*}
	\left(\fint_I w^{1+\eta}(x)dx\right)^{\frac{1}{1+\eta}}\le c_\eta \fint_I w(x)dx.
	\end{equation*} 
	In particular, we have $w^{1+\eta} \in A_1$ and $A_1(w^{1+\eta})\le c_\eta^{1+\eta} A_1(w)^{1+\eta}$.\\
	Furthermore, $c_\eta<c_3:=2e, \ \forall \eta<1$.
\end{thm}
Thank to this result we obtain that for every $w \in A_1$ there is always a $\eta \in (0,1)$ such that $A_1(w^{1+\eta})\le c_3^{1+\eta}A_1(w)^{1+\eta}$.\\
\begin{proof}[Proof of Theorem \ref{coifdec}]
	Choose a sufficiently small $\eta<1$ as in the reverse Holder inequality and define
	\begin{equation}
	b(x)=\frac{w(x)}{[M(w^{1+\eta}(x))]^{\frac{1}{1+\eta}}} \text{ and } g(x)=w(x)^{1+\eta},
	\end{equation}
	noticing that $g(x)\in L^1_{loc}(\mathbb{R})$.\\
	We have that $b(x)\le 1 \iff w^{1+\eta}(x) \le M(w^{1+\eta}(x))$ but this latter inequality holds a.e.since almost every point $x \in \mathbb{R}$ is a Lebesgue point for $w^{1+\eta}(x)\in L^1_{loc}(\mathbb{R})$.\\
	A lower bound on $b(x)$ can be obtained by observing that:
	\begin{equation}
		b(x)=\frac{w(x)}{[M(w^{1+\eta}(x))]^{\frac{1}{1+\eta}}}=\left[\frac{w(x)^{1+\eta}}{M(w^{1+\eta}(x))}\right]^{\frac{1}{1+\eta}}\ge \left[\frac{1}{A_1(w^{1+\eta})}\right]^{\frac{1}{1+\eta}}\ge \frac{1}{c_3A_1(w)}>0.
	\end{equation}
	Now let us take $\varepsilon \in (0,1)$ to be $\frac{1}{1+\eta}$.\\
	We have that $w(x)=b(x)\left[Mg(x)\right]^{\varepsilon}$ so that the theorem is proven and we can choose A as $\frac{1}{c_3A_1(w)}$.
	\end{proof}
By combining theorems \ref{powersofm} and \ref{coifdec}, in light of the duality between $BLO$ and $A_1$ expressed by lemma \ref{dualita} we get that:
\begin{equation} \label{decompo}
f(x) \in BLO \iff \exists\  \alpha>0,\ b(x) \in L^\infty,\ g(x) \in L^1_{loc}, \quad f(x)=\alpha\log(Mg(x))+b(x)
\end{equation}
As a matter of fact, if we have $f \in BLO$, then by lemma \ref{dualita} there is a $\mu>0$ such that $e^{f/\mu}\in A_1$ can be decomposed as in Theorem \ref{coifdec} and the conclusion follows by taking logarithms of both sides.\\ The inverse implication follows from the fact that $(Mg)^\frac{1}{2} \in A_1$ by lemma \ref{powersofm}, which in turn implies $\alpha\log(Mg(x))=2\alpha\log((Mg)^\frac{1}{2})$ is in $BLO$ by lemma \ref{dualita}. The obvious inclusion $L^\infty \subset BLO$ completes this line of reasoning.\\
The logical equivalence in (\ref{decompo}) leads us to define:
\begin{equation} \label{dog}
\|f\|_{BLO}'=\inf\{\alpha+\|b\|_{\infty}:\ \alpha>0,\ b\in L^\infty \ \text{ s.t. }  \exists g \in L^1_{loc}, \ f=\alpha\log(Mg)+b\}
\end{equation}
which has, of course, the same properties of $\|\cdot\|_{BLO}$. In the next section we prove the equivalence of the two.
\section{Equivalence of the norms}
\begin{thm}
	There exist two absolute constants $d_1,d_2>0$ such that, for every $f \in BLO$:
	\begin{equation}
	d_1\|f\|_{BLO}\le  \|f\|_{BLO}' \le d_2\|f\|_{BLO}
	\end{equation}
where $||f||_{BLO}'$ has been defined in \eqref{newnorm}.
\begin{proof} During this proof, $c_1$ and $c_2$ will still denote the constants in equation (\ref{JN}) of lemma \ref{JNLEMMA}.\\We will divide the proof in steps, addressing the latter inequality first.\\
\textbf{Step 1: Showing that $\exists c_4>0$ such that $\forall f \in BLO, \ A_1\left(e^{\frac{fc_2}{4\|f\|_{BLO}}}\right)\le c_4$:}\\
Let us take $\lambda=\frac{4\|f\|_{BLO}}{c_2}\log(\zeta)$ in John-Niremberg inequality (\ref{JN}) from lemma \ref{JNLEMMA}.\\
We have:
	\begin{equation} \label{JNmod}
\Bigg|\, \Bigg\{t \in I:\ |f(t)-f_I|>\frac{4\|f\|_{BLO}}{c_2}\log(\zeta) \Bigg\} \,\Bigg|\le c_1\zeta^{-\frac{4\|f\|_{BLO}}{\|f\|_{BMO}}}|I|.
\end{equation}
Notice that:
\begin{equation} \label{JNmod2}
|f(t)-f_I|>\frac{4\|f\|_{BLO}}{c_2}\log(\zeta) \iff e^{\frac{c_2|f(t)-f_I|}{4\|f\|_{BLO}}}>\zeta.
\end{equation}
By Cavalieri's principle we have:
\begin{multline} \label{ultima}
\int_I e^{\frac{c_2|f(t)-f_I|}{4\|f\|_{BLO}}}dt =\int_0^\infty |\{t \in I:\ e^{\frac{c_2|f(t)-f_I|}{4\|f\|_{BLO}}}>\zeta\}|d\zeta \le\\ \le |I|+\int_1^\infty |\{t \in I:\ e^{\frac{c_2|f(t)-f_I|}{4\|f\|_{BLO}}}>\zeta\}|d\zeta.
\end{multline}
Thank to (\ref{JNmod2}) we are able to bound the last integral so that (\ref{ultima}) becomes:
\begin{equation} \label{ultima2}
\int_I e^{\frac{c_2|f(t)-f_I|}{4\|f\|_{BLO}}} dt \le |I|+|I|\int_1^\infty c_1\zeta^{-\frac{4\|f\|_{BLO}}{\|f\|_{BMO}}} \le |I|+|I|\int_1^\infty c_1\zeta^{-2}d\zeta =|I|(1+c_1).
\end{equation}
where $\|\cdot\|_{BMO}\le 2 \|\cdot\|_{BLO}$ was also used.\\
By dividing both sides of (\ref{ultima2}) by $|I|$ and using $x\le |x|$, we get to
\begin{equation}
\fint_I e^{\frac{c_2(f(t)-f_I)}{4\|f\|_{BLO}}} dt \le 1+c_1
\end{equation} so that:
\begin{equation}
\fint_I e^{\frac{c_2}{4\|f\|_{BLO}}f(t)}dt \le (1+c_1)e^{\frac{c_2f_I}{4\|f\|_{BLO}}}\le (1+c_1)e^{\frac{c_2(\inf_If+\|f\|_{BLO})}{4\|f\|_{BLO}}}
\end{equation}
where we used $f \in BLO$. Lastly:
\begin{equation}
\fint_I e^{\frac{c_2}{4\|f\|_{BLO}}f(t)}dt \le (1+c_1)e^{\frac{c_2}{4}}\inf_Ie^{\frac{c_2}{4\|f\|_{BLO}}f(t)}
\end{equation}
and this first step is concluded by taking $c_4=(1+c_1)e^{\frac{c_2}{4}}$.\\
\textbf{Step 2: Showing that $\|f\|_{BLO}' \le d_2\|f\|_{BLO}$}
Take $f(x) \in BLO$. By the first step we have $w(x)=e^{\frac{f(x)c_2}{4\|f\|_{BLO}}}\in A_1$.\\
By theorem \ref{coifdec} there exist $b \in L^\infty$, $0<A<b(x)<1$, $\varepsilon \in [0,1)$, $g \in L^1_{loc}$ such that
\begin{equation}
w(x)=b(x)[Mg(x)]^\varepsilon
\end{equation}
and so
\begin{equation}
f(x)=\frac{4\|f\|_{BLO}}{c_2}\log(b(x))+\frac{4\|f\|_{BLO}}{c_2}\varepsilon\log(Mg(x)).
\end{equation}
In particular, define 
\begin{center}$B(x)=\frac{4\|f\|_{BLO}}{c_2}\log(b(x))$ and $\alpha=\frac{4\|f\|_{BLO}}{c_2}\varepsilon$.
\end{center}
Since $1\ge b(x) \ge \frac{1}{c_3A_1(w)}$, taking logarithms and using the first step we have that
\begin{equation} \label{ecola} 0<|B(x)|<[\log(c_3)+\log(c_4)]\frac{4\|f\|_{BLO}}{c_2}=:c_5\frac{4\|f\|_{BLO}}{c_2}.\end{equation}
Also notice $\alpha<\frac{4\|f\|_{BLO}}{c_2}$ since $\varepsilon<1$.\\
By combining this with (\ref{ecola}) we get $\|f\|_{BLO}'\le \alpha+\|B\|_{\infty}\le \left[(1+c_5)\frac{4}{c_2}\right]\|f\|_{BLO}$, so that by denoting $d_2=\left[(1+c_5)\frac{4}{c_2}\right]$ we conclude the second step. \\
\textbf{Step 3: Showing $d_1\|f\|_{BLO}\le  \|f\|_{BLO}'$.}\\
Let us choose any $\alpha>0$, $g \in L^1_{loc}$ and $b\in L^\infty$ such that $f=\alpha\log(Mg)+b$.\\
We have:
\begin{equation*}
\|f\|_{BLO}\le \alpha\|\log(Mg)\|_{BLO}+\|b\|_{BLO}\le 2\alpha\|\log([Mg]^{\frac{1}{2}})\|_{BLO}+2\|b\|_{\infty}.
\end{equation*} 
Thank to the first proposition in Lemma \ref{dualita} we have:
\begin{equation*}
\|f\|_{BLO}\le 2\alpha \log\left[A_1([Mg]^{\frac{1}{2}})\right]+2\|b\|_{\infty}.
\end{equation*}
By Theorem \ref{powersofm} we have that $A_1([Mg]^\varepsilon)$ does not depend on $g$, so that $A_1([Mg]^\frac{1}{2})\le c_6$, with $c_6$ absolute constant.\\
For this reason we have
\begin{equation}
\|f\|_{BLO}\le 2\alpha\log(c_6)+2\|b\|_{\infty},
\end{equation}
for each possible decomposition of $f$, so that by taking the infimum and denoting $\frac{1}{d_1}=\max\{2,2\log(c_6)\}$ we conclude the last step of the proof.
\end{proof}
\end{thm}\noindent
\begin{thm}
	For every $f \in BLO$ define:
	\begin{equation*}
	\sigma(f)=\inf\Bigg\{\mu>0: \ e^{\frac{f}{\mu}}\in A_1\Bigg\}
	\end{equation*}
	we have that:
	\begin{equation*}
	dist_{BLO'}(f,L^\infty)=\inf\limits_{h \in L^\infty}\|f-h\|_{BLO}'=\sigma(f)
	\end{equation*}
	\begin{proof}
	First notice that:
	\begin{equation*}
	dist_{BLO'}(f,L^\infty)=\inf\limits_{h \in L^\infty}\inf\{\mu+\|b\|_{\infty}: \ f-h=\mu\log(Mg)+b\}
	\end{equation*}
	so that, for any $\mu$ and $g$, by a proper choice of $h$, $b$ can always be chosen to be $0$.\\
	This means that
	\begin{equation} \label{gat}
	dist_{BLO'}(f,L^\infty)=\inf\{\mu>0:\ \exists g(x)\in L^1_{loc}, \ f(x)-\mu\log[Mg(x)]\in L^\infty\}.
	\end{equation}
	On the other hand, since a weight $e^{\frac{f}{\mu}} \in L^1_{loc}$ is in the Muckenhoupt class $A_1$ if and only if it can be written as $b(x)[Mg(x)]^\varepsilon$ for suitable $0<A<b(x) \in L^\infty$, $g \in L^1_{loc}$ and $\varepsilon \in [0,1)$, we have:
	\begin{equation}
	\sigma(f)=\inf\{\mu>0: \ \exists b \in L^\infty,\ \varepsilon<1,\ g \in L^1_{loc}, \ e^{\frac{f(x)}{\mu}}=b(x)[Mg(x)]^\varepsilon\}
	\end{equation}
	or, by taking logarithms:
	\begin{equation}
	\sigma(f)=\inf\Bigg\{\mu>0: \ \exists b(x) \in L^\infty,\ \varepsilon<1,\ g(x) \in L^1_{loc}, \ \frac{f(x)}{\mu}=b'(x)+\varepsilon\log[Mg(x)]\Bigg\},
	\end{equation}
	where we have denoted $b'(x)=\log(b(x)) \in L^\infty$.\\
	Equivalently:
	\begin{equation}\label{foca}
	\sigma(f)=\inf\{\mu>0: \ \exists \varepsilon<1,\ g(x) \in L^1_{loc}, \ f(x)-\mu\varepsilon\log[Mg(x)]\in L^\infty\}
	\end{equation}
	 is the same as:
	\begin{equation}\label{foca2}
	\sigma(f)=\inf\{\mu>0: \ \exists g(x) \in L^1_{loc}, \ f(x)-\mu\log[Mg(x)]\in L^\infty\}.
	\end{equation}
	Let us better explain this last logical step, between equation (\ref{foca}) and (\ref{foca2}).\\ A real number $a$ satisfies \begin{equation}
	a<\inf\{\mu>0: \ \exists \varepsilon<1,\ g(x) \in L^1_{loc}, \ f(x)-\mu\varepsilon\log[Mg(x)]\in L^\infty\}
	\end{equation}
	if and only if there is no $\varepsilon<1$ and $g \in L^1_{loc}$ such that $f-a\varepsilon\log(Mg(x)) \in L^\infty$, which is equivalent to saying that:
	\begin{equation}
	\nexists a'<a: \ f-a'\log(Mg(x))\in L^\infty
	\end{equation}
	which happens if and only if:
	\begin{equation*}
	a\le \inf\{\mu>0: \ \exists g(x) \in L^1_{loc}, \ f(x)-\mu\log[Mg(x)]\in L^\infty\}.
	\end{equation*}
	The conclusion follows by comparing (\ref{gat}) and (\ref{foca2}).\\
	\end{proof}
\end{thm}

\subsection{Norm attaining property}
In this section, we will show that our new norm has the norm attaining property on $BLO$, i.e. this norm is attained on the closed unit balls of $BLO$. \\
The classical norm on $BLO$ instead has got norm attaining property only on function in $VLO$. Actually, there are examples of $BLO$ functions for whom the norm attaining property does not hold (see \cite{giak}).\\
\begin{mydef}
Let $\Omega \subset \mathbb{R}$. We say that a measurable function $f$ belongs to $EXP(\Omega)$ if there exists a number $\mu \in [0,\infty)$ such that \begin{equation}
	\int_{\Omega} e^{\frac{f}{\mu}} < \infty .
	\end{equation}
	\end{mydef}
\begin{thm} \label{minimum}
		If $f \in BLO$, then $\sigma(f)$ is a minimum, i.e. \\
	$$	\sigma(f)=\min\{\mu>0:\ e^{\frac{f}{\mu}} \in A_1 \} $$ 
\end{thm}
\begin{proof}
Since $f$ belongs to $BLO$, we have that $e^{\frac{f}{\mu}} \in A_1$, and so: \\
$$\sup_{I \subset \subset \mathbb{R}} \frac{\fint_I e^{\frac{f}{\mu}}  dx}{\inf_I e^{\frac{f}{\mu}}} < \infty$$\\
i.e. there exists a constant $C$ such that:$$\sup_{I \subset \subset \mathbb{R}} {\fint_I e^{f-\inf_I f}dx} \le C.$$  \\
	So, to obtain the thesis, we just have to prove that: \\
	\begin{equation} \label{min}
	\inf \Bigg\{ \mu > 0 : \fint e^{\frac{h}{\mu}} \le C\Bigg\}, 
	\end{equation}	
 is a minimum, where $h$ is a  $L^{\infty}$ function. \\
		Throughout this proof we will use $\lambda = 1/\mu$ instead of $\mu$.
	Let us consider $g=e^h$. \\
    Let us suppose that $g$ is not almost everywhere equal to $1$ and
	define $M = \{\sup{\lambda > 0 : g \in L^{\lambda} } \}$.\\
	Let us first suppose that $M < + \infty $. \\
	It is well known that $||g||_{\lambda}$ and hence $||g||_{\lambda}^{\lambda}$ is continuous	with respect to $\lambda$ in $[0,M).$ \\
	We also notice that it is strictly increasing because $g$ is not almost everywhere equal to $1$. \\
	So, there are two possibilities: either $||g||_M^M$ is at most $C$ (and $\frac{1}{M}$ is the minimum we were looking for in the theorem) or $||g||_M^M 
	\in (C,+ \infty]$ but there is unique $M' < M$ such that $||g||_{M'}^{M'}=C$, by the definition of limit of $||g||_{\lambda}^{\lambda} \to ||g||_M^M$ as $\lambda$ goes to $M$,
	monotonicity and the intermediate value theorem.\\
	In either case, the infimum \eqref{min} is attained. \\
	Now let us consider the case that $M = + \infty$. \\
	Of course $\lim_{\lambda \to +\infty} ||g||_{\lambda } = ||g||_{\infty} > 1$
(it is possible that $||g||_{\infty}$ is infinite but that is not a problem) because $h$ is not almost everywhere equal to $0$ but this means that $\lim_{\lambda \to \infty} ||g||_{\lambda}^{\lambda}$. So there
	exists $M' < \infty$ such that  $||g||_{M'}^{M'}=C$, again by definition of limit, monotonicity and intermediate value theorem.  \\
If $g$ is almost everywhere equal to $1$, then $h=0$ a.e., so the minimum in \eqref{min} is attained for every $\mu > 0$.

\end{proof}

\begin{cor} \label{expo}
	If $h \in EXP$, then $\sigma(h)$ is a minimum, i.e. \\
	The quantity \begin{equation}\label{exp}	\inf\Bigg\{\mu>0: \fint_{[a,b]} e^{\frac{h}{\mu}} \le C \Bigg \} 
	\end{equation} is a minimum.
\end{cor}
\begin{proof}
It suffices to use the fact that $f \in L^{\infty}$ in the proof of Theorem \ref{minimum}, so the thesis follows trivially from the proof of Theorem \ref{minimum}
\end{proof} 
\begin{rmk}
	The quantity in \eqref{exp} is a norm on $EXP$. By corollary \ref{expo}, $EXP$ with this norm is norm attaining.
	\end{rmk}

\section{Approximability} \label{approx}
\subsection{Example of an unbounded function in $\overline{L^\infty}^{BLO}$}
It follows from our new distance formula that the closure of $L^\infty$ in $BLO$ can be characterized as:
\begin{equation}
f \in \overline{L^\infty}^{BLO} \iff \forall \mu>0, \ e^{\frac{f}{\mu}} \in A_1. \label{char}
\end{equation}
Using the characterization in (\ref{char}), let us start by providing an example of an unbounded function in $\overline{L^\infty}^{BLO}$, proving the strict inclusion $L^\infty \subset \overline{L^\infty}^{BLO}$ is strict.\\ Throughout this example we restrict the attention to the interval $J=(0,\frac{1}{e})$ and consider the function $f(x)=\log(-\log(x))$, that clearly does not belong to $L^\infty(J)$.\\
However, we will prove $e^{rf(x)}=[-log(x)]^r$ is $A_1$ on $(0,\frac{1}{e})$ for arbitrarily large $r$, concluding through (\ref{char}) that it is in the closure of $L^\infty$ in $BLO$.\\
To prove this, we observe that $g_r(x)=[-log(x)]^{r}$ is decreasing and always greater than $1$ in $J$.
This helps with the extimate of $A_1(g_r)$.\\
Indeed\\
\begin{equation} \label{panegirico}
A_1(g_r)=\sup\limits_{I \subset \subset J} \frac{\fint_I g_r(x)dx }{\inf_Ig_r} =\sup\limits_{[0,b] \subset \subset J} \frac{\fint_0^b g_r(x)dx }{\inf_{[0,b]}g_r} =\sup\limits_{b \le 1/e} \frac{1}{g_r(b)}\fint_0^b g_r(x)dx.
\end{equation} 
This is because $g_r(x)$ is decreasing and for every interval $I=[a,b]$. \\ By considering the interval $I'=[0,b]$ we have the obvious inequalities:
\begin{itemize}
	\item $(g_r)_{I'}\ge (g_r)_I$
	\item $\inf_{I'}g_r=g_r(b)=\inf_Ig_r$.
\end{itemize} 
By taking $b=e^{-t}$ we have
\begin{equation} \label{bizzarro}
A_1(g_r)=\sup\limits_{t\ge 1} \frac{e^t}{t^r} \int_0^{e^{-t}}[-log(x)]^{r}dx.
\end{equation}
We will now show via the induction principle on the integer \emph{r} that the quantity of which we are taking the supremum in (\ref{panegirico}) is decreasing with respect to $t$ for every positive integer $r$. \\
As a matter of fact, the right hand side equals $1+\frac{1}{t}$ for $r=1$ so the base case of the induction principle is satisfied.\\
We will now proceed by induction and assume it is decreasing for $r-1$ and let us use this notation: $$f_r(t)=\int_0^{e^{-t}}[-log(x)]^{r}dx.$$ \\ Using integration by parts we get the recursive formula $$f_r(t)=\frac{t^r}{e^t}+rf_{r-1}(t)$$and so we obtain  $$\frac{e^t}{t^r}f_r(t)=1+\frac{r}{t}\left[\frac{e^t}{t^{r-1}}f_{r-1}(t)\right].$$\\
As the product of two positive decreasing functions is still decreasing, we proved that the right hand side in equation (\ref{bizzarro}) is decreasing in $t$ for every $n$ and this implies:
\begin{equation*}
A_1(g_r)=\fint_{0}^{1/e} [-\log(x)]^rdx
\end{equation*}
because the supremum in \eqref{panegirico} is obtained when $t=1$.\\
One could compute $A_1(g_r)$ exactly with this strategy, but for us it is enough to observe that it is finite for every $r$, i.e. $A_1(e^{rf})<\infty$ for every $r \in \mathbb{N}$ and $f(x) \in \overline{L^\infty}^{BLO}$ by the aforementioned characterization.
\subsection{Approximability by truncation}
In the previous section, we explicitely showed that the inclusion $L^\infty \subset \overline{L^\infty}^{BLO}$ is strict.\\
\begin{mydef}
Let us call $T$ the set of $BLO$ functions $f$ such that,  if we define the truncated functions,
\begin{equation*}
T_kf=\begin{cases}
k & \text{if }  f(x)\ge k\\
f(x)   & \text{if }  -k \le f(x)\le k\\
-k & \text{if } f(x)\le -k
\end{cases}
\end{equation*}
we get $\lim\limits_{k \to \infty} \|f-T_kf|_{BLO}=0$.\\
\end{mydef}
\noindent
We will now also show that both inclusions in \begin{equation} \label{obiettivo} L^\infty \subset T \subset \overline{L^\infty}^{BLO}\end{equation} are strict.\\
Our example of an unbounded function that is approximable by truncation is exactly the one from the previous section. In fact we will show that $f(x)=\log(-\log(x)) \in T$ even if it is unbounded, concluding $L^\infty \subset T$ is strict.\\ To do so, let us notice that:
\begin{equation*}
f-T_kf =\begin{cases}
f(x)-k & \text{if } f(x)\ge k\\
0 & \text{if } f(x)\le k
\end{cases}
\end{equation*}
and since $f-\{f\}_k$ is non-negative, when we compute $\|f-\{f\}_k\|_{BLO}$ we can ignore all intervals having intersection with $\{x: f(x)< k\}=(e^{-e^{k}},e^{-1})$.\\
In other words, $$\|f-T_kf|_{BLO(0,1/e)}=\|f-k\|_{BLO(0,1/e^{e^k})}=\|f\|_{BLO(0,1/e^{e^k})}.$$\\
Let us set $J_k=(0,1/e^{e^k})$. Our goal is to show that $\|f\|_{BLO(J_k)}\to 0$ as $k \to \infty$. In \cite{garcia}, Garcia showed that $e^{\|f\|_{BLO}}\le A_1(e^f)$.\\ In this case, since $e^{f(x)}=-\log(x)$:
\begin{equation*}
{\|f\|_{BLO(J_k)}}\le \log[A_1(-\log(x))]_{J_k}
\end{equation*}
where the subscript $J_k$ denotes the fact that we are considering the $A_1$ constant of $-log(x)$ over $J_k$.\\
With the same strategy as before, we get:
\begin{equation*}
[A_1(-\log x)]_{J_k}=\sup\limits_{t\ge e^k} \frac{e^t}{t} \int_0^{e^{-t}}[-log(x)]dx = \frac{e^{e^k}}{e^{k}}\int_0^{e^{-e^k}}[-log(x)]dx=\frac{1+e^k}{e^k}
\end{equation*}
so that:
\begin{equation*}
0 \le \lim\limits_{k \to \infty} \|f-T_kf|_{BLO} \le 	\lim\limits_{k \to \infty}\log[A_1(-\log(x))]_{J_k} =	\lim\limits_{k \to \infty}\log(1+\frac{1}{e^k})=0.
\end{equation*} proving the strictness of the first inclusion in (\ref{obiettivo}).\\
\subsection{Functions in $\overline{L^\infty}^{BLO}$ which are not approximable by truncation}
There we show that there exists a function which is approximable by bounded functions but not by truncation, i.e. a function $f \in \overline{L^\infty}^{BLO}$ which does not belong to $T$.\\
To do so, we will first show that any function of the type
\begin{equation} \label{tipo}
\eta(x)=\sum\limits_{n=0}^{\infty} n\chi_{[a_{n+1},a_n)} \quad \text{with } a_n \text{ strictly decreasing to } 0
\end{equation} does not belong to $T$ and then we carefully choose $\{a_n\}$ in such a way that $\eta(x)$ is in $\overline{L^\infty}^{BLO}$.\\
To prove the aforementioned claim, let us notice that, for a function $f(x)$ of the type in (\ref{tipo}), the following inequality hold: \\
\begin{equation}
\|f-\{f\}_k\|_{BLO}=\|f\|_{BLO(0,a_k)} \ge 1, \quad \forall k \in \mathbb{N}.
\end{equation}
This is because the $BLO$ norm of a jump function is at least equal to the maximum size of its jumps: one can prove that it is by choosing intervals of the tipe $[a_{n+1},a_n+\varepsilon]$ with sufficiently small $\varepsilon>0$.\\
Our example of a function in $\overline{L^\infty}^{BLO}$ which is not approximable by truncation will then be of the type described in (\ref{tipo}) with the choice $a_n=1/e^{e^n}$ (considered on the interval $J$).\\
Taking $f$ defined as \begin{center}$f(x)=\sum\limits_{n=0}^\infty n\chi_{[1/e^{e^{n+1}},1/e^{e^n})}$ \end{center}
let us compute $A_1(e^{rf})$. Let us notice that $f$ is decreasing:
\begin{equation*}
A_1(e^{rf})=\sup\limits_{b \le 1/e} \frac{\fint_0^b e^{rf(x)}dx}{e^{rf(b)}} \le \sup\limits_{n \in \mathbb{N}} \frac{\fint_0^{a_{n+1}}e^{rf(x)}dx}{e^{rn}}
\end{equation*}
where the second inequality comes from considering $a_{n+1} \le b < a_n$ and overextimating the integral average over $[0,b]$ with the one over $[0,a_{n+1}]$ since $f$ is decreasing.\\
Explicitely computing the integral via the definition we have:
\begin{equation*}
A_1(e^{rf})=\sup\limits_{n \in \mathbb{N}} \frac{1}{a_{n+1}}\sum\limits_{i=n+1}^\infty e^{r(i-n)}[a_i-a_{i+1}]
\end{equation*}and then recalling the choice of $a_n$ we get 
\begin{equation*}A_1(e^{rf})=\sup\limits_{n \in \mathbb{N}}e^{e^{n+1}}\sum\limits_{j=1}^\infty e^{rj}\left[\frac{1}{e^{e^{n+j}}}-\frac{1}{e^{e^{n+j+1}}}\right]=\sum\limits_{j=1}^\infty \frac{1}{e^{[e^n(e^j-e)-rj]}}.
\end{equation*}
Noting that the series on the right hand side is dominated by a convergent one that does not depend on $n$, i.e:
\begin{equation*}
\sum\limits_{j=1}^\infty \frac{1}{e^{[e^n(e^j-e)-rj]}}\le \sum\limits_{j=1}^\infty \frac{1}{e^{[e^j-e-rj]}}=K(r)<+\infty, \quad \forall r,n\in \mathbb{N}
\end{equation*}
shows that $f(x)$ lies in $\overline{L^\infty}^{BLO}$, but, like every other function of this type, is not approximable by truncation.\\
It is worth noticing that one can adapt the proof that $\log(-\log(x))$ is in $T$ to every positive decreasing $VLO$ function, so that all positive decreasing $VLO$ functions are in $T$. 

\subsection{Approximability by convolution}
	In \cite{mio} we also found an equivalent expression for $dist_{BLO}(f,VLO)$ showing in particular that $VLO$ is a closed subset in $BLO$.\\
The procedure we used to construct a function in $VLO$ which was as \textit{close} as possible to a given $BLO$ function was not exclusively that of convolution.\\ A natural question then arises: does the sequence of convolutions of a $VLO$ function $f$ with a standard sequence of mollifiers $\varphi_\varepsilon$ converge (in $BLO$ norm) to the original function $f$ as $\varepsilon \to 0$?\\
We will prove that it does if we choose $\varphi$ with the following properties:
\begin{itemize}
	\item $\varphi_\varepsilon>0$
	\item $supp(\varphi_\varepsilon)\subset (-\varepsilon,\varepsilon)$
	\item $\int_{-\varepsilon}^\varepsilon \varphi_\varepsilon=1$
	\item $\|\varphi_\varepsilon\|_{L^\infty}\le \frac{C}{\varepsilon}$.
\end{itemize} 
\noindent 
Since $VLO$ is closed in $BLO$ and the convolution of a $BLO$ function with a continuous function with compact support lies in $VLO$, this will prove that a function $f \in BLO$ can be approximated by a sequence of its convolutions with a family of mollifiers if and only if $f \in VLO$.\\
Let us call $W_a(f)=\sup\limits_{|I|\le a} \left[\fint_I f-\inf_I f\right]$ and $f_\varepsilon(x)=(f \star \varphi_\varepsilon)(x)$.\\
We have:
\begin{equation*}
f_\varepsilon(x)=\int_{-\varepsilon}^\varepsilon \varphi_\varepsilon(y)f(x-y)
\end{equation*}
and so:
\begin{equation*}
f_\varepsilon(x)-f(x)=\int_{-\varepsilon}^\varepsilon \varphi_\varepsilon(y)[f(x-y)-f(x)]dy\le \int_{-\varepsilon}^\varepsilon \varphi_\varepsilon(y)\left[f(x-y)-\inf_{[x-\varepsilon,x+\varepsilon]}f\right]dy
\end{equation*}
Let us multiply and divide by $2 \varepsilon$ to get:
\begin{equation*}
|f_\varepsilon(x)-f(x)|\le 2C\fint_{-\varepsilon}^\varepsilon f(x-y)-\inf_{[x-\varepsilon,x+\varepsilon]}dy \le 2CW_{2\varepsilon}(f).
\end{equation*}
But the right hand side does not depend on $x$ anymore so that:
\begin{equation*}
0 \le \|f-f_\varepsilon\|_{BLO}\le 2\|f-f_\varepsilon\|_{L^\infty} \le 4CW_{2\varepsilon}(f) \to 0
\end{equation*}
and the proof is concluded.

\section{Duality}
As we said in the introduction, thanks to \cite{feff} we know that the predual of $BMO$ is the Hardy space $H^1$, but it is still unknown who is $BMO^{*}$. To identify the dual space of $BMO$ is a very hard and interesting question.\\
Obviously, since $BLO$ and $VLO$ are not vectorial spaces, we can not talk about their dual in classical sense, and we define the following classes of linear and bounded functionals and we call them "dual" the same: \\
\begin{mydef}
Let us call "dual" of $BLO$ the following class: \\
\\
$BLO^{\ast}=\{\phi: \phi \, \text{is a linear functional defined on} \, BLO \, \text{such that} \, |\phi(f)| \le K ||f ||_{BLO} \}$\footnote{i.e. they respect $\phi(au + bv) = a\phi(u) + b\phi(v)$ for every $a,b,u,v$ such that both sides of the
	equality make sense}
\end{mydef}
                                 $$\text{and}$$
\begin{mydef}
	Let us call "dual" of $VLO$ the following class: \\
	\\
	$VLO^{\ast}=\{\phi: \phi \, \text{is a linear functional defined on} \, VLO \, \text{such that} \, |\phi(f)| \le K ||f ||_{BLO} \}$.
\end{mydef}
We are able to prove the following result: \\
\begin{thm}
 $V LO^{\ast}$ is isometric to $VMO^{\ast}$.
\end{thm}
\begin{proof} Let us observe that for $\psi \in BMO^{\ast}$ we have that $\psi_{|VLO}$ is still linear on $VLO$ and 
$$\psi_{|VLO}(f) \le ||\psi||_{VMO}^{\ast} ||f||_{BMO} \le ||\psi||_{VMO^{\ast}} 2|| f ||_{BLO}.$$
It means that $\psi_{|VLO}$ is continuous as a functional on $VLO$, with operator norm at
most doubled in the process.\\
Let us now observe that for every $\phi \in VLO^{\ast}$, we can define $\hat{\phi}$ acting on $VMO$ by using
a decomposition theorem from Korey \cite{KOR}. \\ Actually, M.B. Korey showed that for every
$g \in VMO$, there exist $g_{1}, g_{2} \in VLO$ such that $g = g_{1} - g_{2}$ and $g_{1},g_{2}$
can be chosen such that
\begin{equation} \label{dis}
||g_{1}||_{BLO} +||g_{2}||_{BLO} \le ||g||_{BMO}.
\end{equation}
 \noindent
Inspired by this observation we define
$\hat{\phi}(g) = \phi(g_{1})-\phi(g_{2})$. \\
First, we have to prove that the definition is well posed, because the decomposition
is not unique.\\
So let us notice that:\\ 
 \begin{multline}
	g_1 - g_2 = h_1 - h_2 \Rightarrow g_1 +h_2 = g_2 + h_1 \Rightarrow \phi(g_1 + h_2) = \phi (g_2 + h_1) \Rightarrow
\phi(g_1) + \phi (h_2) = \\ = $$ \phi (g_2) + \phi (h_1) \Rightarrow \phi (g_1)-\phi(g_2) = \phi(h_1)-\phi(h_2)$$
\end{multline}
so that $\hat{\phi}$ is well defined on $VMO$ (in the sense that it does not depend on the
decomposition).\\
Of course $\hat{\phi}$  is linear on $VMO$ and using \eqref{dis} we also obtain that it is bounded, with
operator norm at most equal to the operator norm of $\phi$.\\
In formulas: \\
$$|\hat{\phi}(g)| \le |\phi(g_1)|+|\phi(g_2)| \le  ||\phi||_{VLO^{\ast}} (||g_1||_{BLO} + ||g_2||_{BLO} \le ||\phi||_{V LO^{\ast}} ||g||_{BMO}.$$ \\
We now notice that ${\hat{\psi}}_{|VLO} = \psi$ for all $\psi \in  VMO^{\ast}$
and $\hat{\phi}_{|VLO} = \phi$ for all $\phi$ in
$VLO^{\ast}$,
so that we constructed a bijection between $VLO^{\ast}$
and $VMO^{\ast}$. \\
Since the action to put the hat on a function that belongs to $BLO$  and
the restriction are linear actions it is an isomorphism of vector spaces and by the
fact we were able to control the operator norm in the process it is an isometry.
\end{proof}

\begin{rmk}
$VMO^{\ast}$ is isomorphic to the Hardy space $H_1$ (see \cite{stein} p.180).
\end{rmk}
\section{Acknowledgements}
\noindent 
The author is member of the Gruppo Nazionale per
l’Analisi Matematica, la Probabilità e le loro Applicazioni (GNAMPA) of the
Istituto Nazionale di Alta Matematica (INdAM). \\
The author would like to warmly thank Professor Carlo Sbordone for having suggested her studying this problem. \\

\end{document}